\documentclass{article}

\usepackage{amsmath}
\usepackage{amsthm}
\usepackage{graphicx} % Required for inserting images
\usepackage{amsfonts}
\usepackage{latexsym}
\usepackage{mathtools}
\usepackage{hyperref}
\usepackage{comment}
\usepackage{mathrsfs}
\usepackage{stmaryrd}
\usepackage{thm-restate}

\newtheorem{thm}{Theorem}[section]
\newtheorem{lma}{Lemma}[section]
\newtheorem{exa}{Example}
\newtheorem{prop}{Proposition}[section]
\newtheorem{defi}{Definition}
\newtheorem{rmk}{Remark}

\newcommand{\nat}{\mathbb{N}}
\newcommand{\natp}{\mathbb{N}\setminus\{0\}}
\newcommand{\integer}{\mathbb{Z}}
\newcommand{\deq}{\coloneq}
\newcommand{\pset}{\mathcal{P}}
\newcommand{\qset}{\mathcal{Q}}
\newcommand{\real}{\mathbb{R}}
\newcommand{\lfl}{\lfloor}
\newcommand{\rfl}{\rfloor}
\newcommand{\cali}{\mathcal{I}}
\newcommand{\undPT}{\underline{\pset_T}}
\newcommand{\overPT}{\overline{\pset_T}}
\newcommand{\cale}{\mathcal{E}}
\newcommand{\calf}{\mathcal{F}}
\newcommand{\emptyline}{\vskip\baselineskip}

\newcommand{\prmx}{\mathbb{P}_x}

\newcommand{\elx}{(\log x)}

\newcommand{\rset}{\mathcal{R}}
\newcommand{\sset}{\mathcal{S}}
\newcommand{\tset}{\mathcal{T}}

\newcommand{\ppx}{\mathbb{P}_x}

\newcommand{\eps}{\epsilon}

\newcommand{\sumcirc}{\sum^\circ}

\title{On weird numbers with high abundancy index}

\date{\today}

\begin{document}

\author{Kei Hisamoto}

\maketitle
\begin{abstract}
    Let $n$ be a natural number.
    It is proved that if the sum of divisors of $n$ has sufficiently large relative to $n$, then $n$ is some distinct sum of proper divisors of $n$.

\end{abstract}
\section{Introduction}
 A natural number $n$ is called a $\mathit{semiperfect}\ \mathit{number}$ if there exists a set of proper divisors that sums to $n$.
The $\mathit{abundancy}$ $\mathit{index}$ of $n$ is defined to be $\sigma(n)/n$
 where $\sigma(n)$ denotes sum of all divisors of $n$.

It is conjectured by Benkoski and Erd\H{o}s in \cite[p.617]{BeEr74},\cite[p.47]{erdos1996some},\cite{Er74b} that there is an absolute constant $C$ such that if $n$ has abundancy index larger than $C$, then $n$ is a semiperfect number.
If $n$ has abundancy index greater than 2 and is not semiperfect, then $n$ is called a $\mathit{weird}\ \mathit{number}$.
See also \cite{Erdos825} for further information.

Our main theorem in this paper is :

\begin{thm}\label{mainthm633}
If the abundancy index of a natural number $n$ is larger than $633$, then $n$ is semiperfect.
\end{thm}

\begin{rmk}
    It is expected that $C=3$ is suffice.
\end{rmk}

\begin{rmk}
Desmond Weisenberg observed that a $\mathit{weird}\ \mathit{number}$ with  abundancy index greater than $4$
must have an \textbf{odd weird number} as a divisor.
On the other hand, Erd\H{o}s asked 
``Are there any $\mathrm{odd}$ weird numbers?''. see $\cite{Erdos470}$,$\cite{Erdos825}$.
\end{rmk}

The outline of the proof of Theorem \ref{mainthm633} is as follows:
The first step in our method is to show that if $n$ has a large   abundancy index, then the set $\pset$ of primes dividing $n$ is large in the sense that the sum of reciprocals of its elements is larger than a certain constant.
The second step is to find ``good'' subset
$\mathcal{P}'\subset \mathcal{P}$ by omitting a ``thin'' part and  ``too small'' primes in $\pset$. $\pset'$ satisfies the conditions in Theorem \ref{logicallymainthm}.

Theorem \ref{logicallymainthm} implies that if a set of prime numbers $\qset$ only has large elements and do not have ``thin'' part, the product of all elements of $\mathcal{Q}$  is  semiperfect.

The proof of Theorem \ref{logicallymainthm} is as follows:
Assume that a set of primes $\overline{\pset}$ satisfies the conditions in Theorem\ref{logicallymainthm}, find subset $\pset_{g}\subset \overline{\pset}$ such that distribution of divisors of $N''\deq\Pi_{p\in\pset_g}p$ satisfies certain conditions. ($\pset_g$ is made by ommitting ``too large'' primes.) These conditions are strong enough to imply that $N''$ is semiperfect.

In the argument above, we used the lemma below.

\begin{lma}\label{spdivisor}
Assume $n\in\nat$ and let $n'$ be a divisor of $n$.
Then if $n'$ is semiperfect, $n$ is semiperfect.
\end{lma}
\begin{proof}
Assume $n'$ is semiperfect. Let $n'= d_1+\cdots+d_m$ be a representation of $n'$ as a sum of distinct proper divisors of $n'$. 
Then $n=d_1\cdot n/n'+\cdots+d_m\cdot n/n'$ is a representation of $n$ as the sum of distinct  proper divisors of $n$.
\end{proof}

\section{Notation and some examples}

Let $A$,$B$ and $A_i\ (i \in I)$ 
be subsets of $\integer$,
where $I$ is finite index set.
Define
\begin{equation}
A+B \coloneq \{a + b\ |\ a\in A,b\in B \}\ ,\notag
\end{equation}
\begin{equation}
AB\deq\{ab\ |\ a\in A,b\in B \}\ .\notag
\end{equation}
When $A = \{a\}$ is a set with one element, we write
\begin{equation}
a+B \coloneq\{a\}+B = \{a + b\ |\ b\in B \}\ , \notag
\end{equation}
\begin{equation}
aB \coloneq\{a\}B = \{ab\ |\ b\in B \}\ . \notag
\end{equation}
The definition is similar when there are finitely many sets:
\begin{equation}
\sum_{i = 1}^{n}A_i \coloneq \{a_1+a_2+\dots + a_n\ |\ a_i\in A_i \}\ . \notag
\end{equation}
For $0\neq a\in \nat$, define $I(a) \coloneq \{0,1,\dots,a\} $
and $[a]\coloneq \{0,a \}$. Note that $I(1)=[1]$.
We see that a natural number $n \in \nat$ is semiperfect if and only if
\begin{gather}
n\in\sum_{d|n,d\neq n}[d] =\sum_{i=1}^t\{0\cdot d_i ,1\cdot d_i \}  \notag
\\ =\big\{\sum_{i=1}^t \epsilon_i d_i | \epsilon_i\in\{0,1\}\big\}
. \notag
\end{gather}
Let $a,b,c\in\nat$. By $a\prec b$, we mean $b\le 2a$.
By $a\prec b\prec c$, we mean $a\prec b$ and $b\prec c$.
For $X\subset\nat$, define
\begin{equation}
\Pi(X) \deq \prod_{x\in X}x\ , \notag
\end{equation}
\begin{equation}
\displaystyle \sum (X) \deq\sum_{x\in X}x\ . \notag
\end{equation}
For $k\in\nat$, define
\begin{equation}
\binom{X}{k} \deq\{ S\ \big|\ S\subset X, |S| = k \}\ , \notag
\end{equation}
\begin{equation}
\binom{X}{k}_\Pi \deq \{\Pi(S)\ \big|\ S\in\binom{X}{k}\}\ . \notag
\end{equation}

\begin{exa}\label{plusminus}
Let $a,b\in\nat$, and assume $a < b$ then we have
\begin{equation}
[a]+[b] = \{0,a\}+\{0,b\}=\{0,a,b,a+b\}\supset
\{a,b\}=\{a\}+[b-a] ,\notag
\end{equation}
\begin{equation}
[a] + [b]\supset[a+b] .\notag
\end{equation}
\end{exa}
\begin{exa}\label{extinterval}
Let $a,b\in \natp$, $b\leq a +1$.
\begin{equation}
I(a)+[b] = I(a) + \{0,b\} = I(a + b). \notag
\end{equation}
\end{exa}

\section{Preliminaries}
We call a sequence  $a_1,a_2,\dots,a_n$ of natural numbers
$\mathit{slowly}\ \mathit{growing}\ \mathit{sequence}$ (SGS) when 
\begin{equation}
a_k\leq a_1 + a_2 +\dots+a_{k-1}+1 \notag
\end{equation}
for all $2\leq k\leq n$.
\begin{lma}\label{intervalfromsgs}
If $a_1,a_2,\dots,a_n$ is a slowly glowing sequence, 
then 
\begin{equation}
I(a_1) + \sum_{i=2}^n [a_i] =  I\Bigl(\sum_{i=1}^n a_i\Bigr). \notag
\end{equation}

\end{lma}
\begin{proof}
This can be shown by applying Example 2 repeatedly.
\end{proof}
\begin{exa}
We will show that $770 = 2\cdot5\cdot7\cdot11$ is semiperfect.
\begin{gather}
\notag
\sum_{d|770,\ne770}[d] \\ \notag
=[1]+[2]+[5]+[7]+[11]+\dot{[10]}+\dot{[14]}+[22]+[35]+[55]+[77]\\ \notag
+[70]+[110]+[154]+[385]\\ \notag
\supset
10+I(1)+[2]+[4]+[5]+[7]+[11]+[22]+[35]+[55]\\ \notag
+[70]+[77]+[110]+[154]+[385]\\ \notag
= 10+I(938).
\end{gather}
 $1,2,4,5,7,\dots,154,385$ is SGS. 
 The final equality sign follows from Lemmma $\ref{intervalfromsgs}$.
In particular $770$ is semiperfect.

\end{exa}

\begin{lma}\label{sim}
Let $n\ge2$. If $a_1,a_2,\dots,a_n$ is a slowly glowing sequence and 
\linebreak$a_n\prec a_{n+1}$. Then  $a_1,a_2,\dots,a_n,a_{n+1}$ is also an SGS.

\end{lma}
\begin{proof}
By assumption,
$a_k\leq a_1 + a_2 +\dots+a_{k-1}+1$ for all $2\leq k\leq n$.
and
$a_{n+1}\le a_n + a_n\le a_1+\dots+a_{n-1}+1+a_n $. So 
$a_1,a_2,\dots,a_n,a_{n+1}$ is a slowly growing sequence.
\end{proof}

In the rest of this paper,
fix $L \deq 8.371\times10^{45} $.
 
\begin{lma}\label{inequalities}
Under assumptions $\real \ni x>L$ and $y \in \nat $,
following inequalities hold:
\begin{equation}\label{2592log18}
x>2592(\log x)^{21},
\end{equation}

\begin{equation}\label{binom2x}
\binom{y}{k}>2x\ if\ y,k\in\nat, \ y \ge \lceil x/(\log x)^2 \rceil\ ,\ 2\le k\leq y-2,
\end{equation}

\begin{equation}\label{powofx/l4}
\Bigl( 1+\frac{1}{2(\log x)^2}\Bigr) ^{x/(\log x )^2}
>2^{x/2(\log x)^4}
>2x,
\end{equation}
\begin{equation}\label{logxlogm}
\log x > \lfl \log_2(2(\log x)^2)\rfl + 10.
\end{equation}

\end{lma}
\begin{proof}
We will prove first inequality.
We define $f(x)$ by
\begin{equation}
    f(x)\deq \log\biggl(\frac{x}{2592(\log x)^{21}}\biggr). \notag
\end{equation}
We can see $f(L)>0$ and $f'(x)=((\log x) - 21)/(x\log x)>0 \ (x\ge L).$
it means that $f(x)>0\ (x>L)$ and the desired inequality.

Second inequality follows from the special case $k=2$
and the fact that binomial coefficient becomes bigger the more it approaches to $k=\lfl n/2\rfl$.
Third, Fourth inequality comes from first inequality.
\end{proof}

\emptyline

By $\mathbb{P}$ we mean the set of all primes,
and $\mathbb{P}_x\deq\mathbb{P}\cap[x,2x]$.

\begin{restatable}{thm}{fstep}\label{first step}
Let $L<x\in \real$. Assume $\qset\subset[x,2x]$ is a set of prime numbers such that 
\begin{equation}|\qset|>\frac{2x}{(\log x)^2}+10
\end{equation}
and define
\begin{equation}
\mathscr{D}\deq \qset\cup p\cdot (\qset\setminus\{p\})\notag
\end{equation}
where $p$ is $\min(\qset)$.  Then there exist $E,F\in\nat$ such that
$E<4x^3$ and \linebreak$F>4x^2$, such that
\begin{equation}\label{concthm2}
E+I(F)\subset\sum_{d\in \mathscr{D}}[d].
\end{equation}
\end{restatable}
This theorem follows from theorem below:
\begin{restatable}{thm}{primeint}\label{primestointerval}
    Asssume $\pset\subset \prmx $ and 
    \begin{equation}
        |\pset| > \frac{2x}{(\log x)^2} + 9.\notag
    \end{equation}
    Then there exist $\mathcal{E},\mathcal{F}\in \nat$ such that $\mathcal{E}<2x^2$ , $\mathcal{F}>4x$
    such that 
    \begin{equation}
        \mathcal{E}+I(\mathcal{F})\subset \sum_{p\in \pset}[p].\notag
    \end{equation}
\end{restatable}
\emptyline
The concept of $\mathit{multiset}$ is used in the proof of this theorem.
Multiset is set with information of multiplicity.
In this paper, multiset is a function
$\nat$ to $\nat$
\begin{equation}
    \mathcal{S}:a\to \mathrm{mult}_\mathcal{S}(a) \notag
\end{equation}
In this situation
we say $\mathcal{S}$ has $a$ with multiplicity $\mathrm{mult}_\mathcal{S}(a)$.
For ordinally set of natural number $X\subset\nat$,
By $\{X\}^1$ we mean multiset which include elements of $X$ with multiplicity $1$,
$\mathrm{i.e}$
\begin{equation}
    \{X\}^1(n)\deq 
    \begin{cases}
        1 & \text{$n\in X$} \\
        0 & \text{otherwise}. 
    \end{cases}\notag
\end{equation}
We call multiset $\mathcal{S}$ is finite if there is only finite natural number $a\in\nat$ such that $\mathrm{mult}_\mathcal{S}(a)\neq 0$.
In this paper. We only deal with finite multisets.

For convinience, we write multiset in the form of $\{\!\{e_1*\mathrm{mult}_\sset(e_1),\dots,e_z*\mathrm{mult}_\sset(e_z)\}\!\}$.
For example, multiset $\{\!\{2*3,5*2\}\!\}$ means multiset which include 2 with multiplicity 3 and 5 with multiplicity 2, and we drop out $*1$ for convinience. For example, $\{\!\{2,3,5\}\!\}$ means multiset which includes $2,3,5$ with multiplicity $1$.

For two multiset $\mathcal{S}\deq \{\!\{a_1*e_1,\dots,a_l*e_l\}\!\}
,\mathcal{T}$
we define 
$\mathcal{S}+\mathcal{T}$ by

\begin{equation}
    \mathrm{mult}_{\mathcal{S}+\mathcal{T}}(a)
    \deq \mathrm{mult}_\mathcal{S}(a)+
    \mathrm{mult}_\mathcal{T}(a)
    \ (a\in\nat). \notag
\end{equation}
By $FM$ we mean the set of finite multiset on $\nat$.
We can define preorder over $FM$.
We say  $\mathcal{S}\vdash\mathcal{T}$ when
there exist $e_{\mathcal{S},\mathcal{T}}\in\nat$
such that $\sum^\circ(\mathcal{S})\supset  e_{\mathcal{S},\mathcal{T}}
+\sum^\circ(\mathcal{T})$.
this $\vdash$ satisfies transitive property.
Example\ref{plusminus} means that
\begin{gather}
    \{\!\{a,b\}\!\}\vdash\{\!\{a+b\}\!\},\label{abtoab} \\
    \{\!\{a,b\}\!\}\vdash\{\!\{b-a\}\!\}\ (b>a), \label{abtoba} \\
    \{\!\{a_1,\dots,a_L,b\}\!\}\vdash\{\!\{b-\sum_{i=1}^L a_i\}\!\}\ (b> \sum_{i=1}^L a_i), \label{bminussuma} \\
    \{\!\{a_1,\dots,a_{L_1},a_{L_1+1},\dots,a_{L_2},\dots\dots,a_{L_K},b_1,\dots ,b_k\}\!\} \notag \\
    \vdash \{\!\{b_1-\sum_{i=1}^{L_1}a_i\ ,b_2-\sum
_{i=L_1+1}^{L_2}a_i\ ,\dots,b_K-\sum_{i=L_{K-1}+1}^{L_K} a_i\}\!\}\, \label{bbminussuma}  \\
(b_j>\sum_{i=L_{j-1}+1}^{L_j}a_i\ ,\ L_0\deq0) \notag \\
\{\!\{ a*m,b \}\!\}\vdash\{\!\{b-ma\}\!\}\ (b>ma). \label{bmma}
\end{gather}
\begin{lma}\label{sttost}

    Let $\sset,\sset^{'},\tset$ be multisets such that $\sset\vdash\sset^{'}$.
    then
    \begin{equation}
        \sset+\tset\ \vdash\ \sset^{'}+\tset.
    \end{equation}
\end{lma}
\begin{proof}
    This lemma forrows from the fact that if $A\subset A',B$ are set of integer then $A+B\subset  A'+B$.
\end{proof}
By $\mathbb{P}$ we mean the set of all primes,
and $\mathbb{P}_x\deq\mathbb{P}\cap[x,2x]$.

\begin{lma}\label{rtov}
    Let $\mathcal{R}\subset\ppx$ such that $|\rset|>x/\elx^2+1$ then there exist $2\le\nu\le2\elx^2$ such that
    \begin{equation}
    \{\rset\}^1\vdash\{\!\{\nu*\lceil\frac{x}{8\elx^4}\rceil\}\!\}
    \end{equation}
\end{lma}
\begin{proof}
Let $\rset = \{r_1,r_2,\dots,r_m\}\ (r_1<r_{2}<\cdots<r_m)$
%\ m\ge \lceil x/(\log x)^2\rceil+1)% 
And 
\begin{equation*}
g_k \deq r_{k+1} -r_k\ (1\le k\le m-1).
\end{equation*}
We can see the following;
\begin{equation}
\sum_{k = 1}^{m-1}g_k = r_m-r_1<x. \notag
\end{equation}
so, the number of $k$ that satisfy $g_k\ge 2(\log x)^2$
is smaller than $x/2(\log x)^2$. 
Therefore, the number of $1\le k\le m-1$ 
that satisfy $g_k<2(\log x)^2$ is larger than $x/2(\log x)^2$.

By the pigeonhole principle, there exists $0<M<2(\log x)^2$
such that
\linebreak$G_M\deq\{k\ |\ g_k = M,\ 1\le k\le m-1\}$ has a cardinality larger than $x/4(\log x)^4$.

At least one of the two sets $G_M^\mathrm{o}\deq\{k\in G_M|\ \mathrm{odd}\}$, and $\ G_M^\mathrm{e}\deq\{k\in G_M|\ \mathrm{even}\}$ has a cardinality greater than $x/8(\log x)^4$. We may assume $|G_M^\mathrm{o}|> x/8(\log x)^4$. (When $|G_M^e|> x/8(\log x)^4$, just replace $G_M^\mathrm{o} $ in the argument below by $G_M^e$.)
we can see;
\begin{gather}
    \{\rset \}^1=\{\!\{r_1,r_2,\dots,r_m\}\!\} \notag \\
    =\{\!\{r_1,r_2\}\!\}+\{\!\{r_3,r_4\}\!\}+\cdots\{\!\{r_{2\lfl m/2 \rfl-1},r_{2\lfl m/2 \rfl}\}\!\} \notag \\
    \overset{\ref{abtoba}}{\vdash}\{\!\{g_1\}\!\}+\{\!\{g_3\}\!\}+\cdots+\{\!\{g_{2\lfl m/2 \rfl-1}\}\!\} \notag \\
    \vdash \{\!\{M*|G_M^\mathrm{o}|\}\!\} \notag  \\
    \vdash \{\!\{M*\lceil x/8\elx^4\rceil\}\!\}.\notag
\end{gather}
\end{proof}

\begin{lma}\label{stote}
    Let $\mathcal{S}\subset\ppx$ such that $|\sset|>x/\elx^2$ and fix $2\le\mu\le\elx^2$. Then
    there exists a number sequence $t_0,\dots,t_{m}\ (m\deq\log_2(\lfl\mu-0.1\rfl))$ and $e_1,\dots,e_S$ such that
    $0<e_i\le x/8\elx^5\ ,\ \mu|e_i\ ,\ \sum e_i>x/648\elx^{18}$ for $e_1,\dots,e_S$,
     $t_\bullet\le2(\mu-1)x\ ,\ t_\bullet\equiv2^\bullet\mod \mu$ for
     $t_0,\dots,t_m$ and
    \begin{equation}
        \{\sset\}\vdash\{\!\{t_0,\dots t_m,e_1,\dots e_S\}\!\}
    \end{equation}
    are satisfied.
\end{lma}
\begin{proof}
    For $0\le v <\mu $, let
$\sset_v\deq\{r\in\rset\ |\ r\equiv v\mod \mu\}$.
 If $\sset_{x}\ne \emptyset $, then
$\overline {x}\in(\integer/M\integer)^{\times}$ because $\sset\subset\ppx$ is a set of prime numbers larger than $\mu$.
In particular $\sset_{0}=\emptyset$.

By the pigeonhole principle, there exists
Let $1\le w\le  \mu-1$ such that $|\sset_{w}|\ge \lceil x/2(\log x)^4\rceil$.

Divide $[x,2x]$ into $\lceil 8(\log x)^5\rceil+1$
intervals $I_k\ (1\le k\le\lceil 8(\log x)^5\rceil+1)$ whose length is shorter than $x/8(\log x)^5$. At least one of them has more than
\begin{equation}
R\deq\Bigl\lceil\frac{x}{17 (\log x)^9}\Bigr\rceil
\Bigl(<\frac{x}{2(\log x)^4}\cdot\frac{1}{\lceil 8(\log x)^5\rceil+1}\Bigr). \notag
\end{equation}
elements of $\rset_{w}$.
Assume that $D'\deq I_m\cap\rset_{w}$ has cardinality larger than 
$R$.
Construct $D=\{d_1,d_2,\dots,d_R\}\subset D'\ (d_1<d_{2}<\dots<d_R)$ by choosing elements of $D'$.
We write $\sset_w\setminus D =\{s_1,\dots,s_Q\} (s_1<s_2<\dots<s_Q)$.
there is number

Let $e_1\deq d_R-d_1,e_2\deq d_{R-1}-d_{2},\dots,
e_{\lfloor R/2 \rfloor} \deq d_{R-\lfloor R/2 \rfloor+1}-d_{\lfloor R/2 \rfloor} $.
All elements of this set are multiples of $M$ and
\begin{equation}\label{deltaestimate}
\sum_{i=1}^{\lfl R/2\rfl}  e_i>(R-1)^2/2>x^2/648(\log x)^{18}.
\end{equation}
(consider that all elements in $D$ are $\mathrm{odd}$ primes and $e_k\ge2R+2-4k$.)
We see that all elements in $\Delta$ are smaller than $x/8(\log x)^5$.

The cardinality of $\sset_{w}\setminus D$ is larger than
\begin{equation}\label{cardofq2wdash}
\frac{x}{2(\log x)^4}\ - R\ >(\mu-1)\cdot (m + 1).
\end{equation}
\emptyline

Since $\overline{w}\in(\integer/\mu\integer)^\times$, there exists $1\le u_k\le \mu-1\ (0\le k \le m)$ such that
$w\cdot u_k\equiv 2^k \pmod \mu $. Let 
\begin{equation}
U_0\deq0\ ,\ U_k\deq\sum_{i=1}^k u_i\ ,\ t_k\deq\sum_{i=U_{k-1}+1}^{U_k}s_i.
\end{equation}
(\ref{cardofq2wdash} means $U_m<Q.$)
We can see $t_i\equiv w\times u_i\equiv 2^i \mod \mu$
and $t_i<(M-1)\cdot 2x\ (0\le i\le m)$.
Our goal is to gain ().
\begin{gather}
    \{\sset\}\vdash \{\sset_w\}=\{\sset_w\setminus D\}+\{D\} \notag \\
    \vdash \{\!\{s_1,\dots,s_Q\}\!\}+\{\!\{ d_1,\dots,d_R  \}\!\} \notag  \\
    \vdash \{\!\{t_1,\dots ,t_m,e_1,\dots,e_S   \}\!\}  \notag
\end{gather}

\end{proof}

\begin{lma}\label{mutomu1}
Let $2\le\mu\le \elx^2$ and $m\deq\lfl\log_2 (\mu -0.1)\rfl$.
Let $\tau_0,\dots,\tau_m$ be number sequence which satisfies
\begin{equation}
0<\tau_i\le 2(\mu-1)x\ ,\ \tau_i\equiv2^i \mod\mu.
\end{equation}
And let $\eps_1,\dots\eps_T$ be number sequence which satisfy 
\begin{gather}
\mu|\eps_i\ ,\ 0<\eps_i\le\lceil\frac{x}{8\elx^5}\rceil \notag \\
\sum_{i=1}^T e_i>2(m+1)(\mu-1)x + 4x. \notag
\end{gather}
Then there exist $V<T$ such that $\sum_{i=V}^T\eps_i>4x$ and
    \begin{gather}
        \{\!\{\mu*\lceil\frac{x}{8\elx^4}\rceil,\tau_0,\dots,\tau_m,\eps_1,\dots,\eps_T \}\!\} \notag \\
        \vdash\{\!\{ \mu*\lceil\frac{x}{8\elx^5}\rceil,1,2,\dots,2^m,\eps_V,\dots\eps_T\}\!\}
    \end{gather}
    are satisfied.
\end{lma}
\begin{proof}
There exists $j_0$ such that
\begin{equation}
    \sum_{i=1}^{j_0}\eps_i\le\tau_0\le\sum_{i=1}^{j_0 +1}\eps_i\ 
\end{equation}
because $\sum_{i=1}^T \eps_i > 2(m+1)(\mu-1)x+4x>\sum_{i=0}^m \tau_i\ge\tau_0.$
We can repeat this kind of argument recursively.
Assume $(j_{-1}\deq 0,)j_0,\dots,j_K\ (0\le K<m)$ such that 
\begin{equation}
    \sum_{i=j_{k-1}+1}^{j_k}\eps_i\le\tau_k\le \sum_{i=j_{k-1}+1}^{j_k +1}\eps_i\  (0\le k\le K)
    \end{equation}
is satisfird.
Then there exists $j_{K+1}$ such that
\begin{gather}
    \sum_{i=1}^{j_{K+1}}\eps_i\le\tau_{K+1}\le\sum_{i=1}^{j_{K+1} +1}\eps_i
    \end{gather}
because 
\begin{gather}
    \sum_{i=j_K+1}^{T}\eps_i
    =\sum_{i=1}^{T}\eps_i -\sum_{k=0}^{K}\big(\sum_{j_{k-1}+1}^{j_k}\eps_i\big) \notag \\
    >2(m+1)(\mu-1)x+4x-\sum_{k=0}^K \tau_k \notag \\
    \ge2(m+1)(\mu-1)x+4x-2(K+1)(\mu-1)x \notag \\
    \label{remeps}   
    \ge 2(\mu-1)x +4x>\tau_{K+1}
\end{gather}
We can repeat argument above for $K=0,1,\dots,m-1$ to gain number sequence $j_0,\dots j_m$ such that 
\begin{gather}
    \sum_{i=j_{k-1}+1}^{j_k}\eps_i\le\tau_k\le \sum_{i=j_{k-1}+1}^{j_k +1}\eps_i\  (0\le k\le m), \notag \\
    0\le(\rho_k\deq)\ \tau_k-\sum_{i=j_{k-1}+1}^{j_k}\eps_i \le\eps_{j_k+1}\le \lceil\frac{x}{8\elx^5}\rceil   \notag \\
    \sum_{i=j_m +1}^T \eps_i>4x
\end{gather}
are satisfied (Third inequality comes from substitute $m$ to $K$  in culculation \ref{remeps}).$\rho'_i\deq\rho_i-\mu\cdot\lfl\rho_i/\mu\rfl=2^i$
because $0\le\rho'_i<\mu$ and $\rho'_i\equiv\rho_i\equiv2^i \mod\mu.$
At last,
\begin{gather}
    \{\!\{\mu*\lceil\frac{x}{8\elx^4}\rceil,\tau_0,\dots,\tau_m,\eps_1,\dots,\eps_T\}\!\} \notag \\
\vdash\{\!\{\mu*((m+2)\cdot \lceil\frac{x}{8\elx^5}\rceil),\tau_0,\dots,\tau_m,\eps_1,\dots,\eps_{j_0}, \notag \\
\eps_{j_0 +1},\dots,\eps_{j_1},\dots\dots\eps_{j_{m-1}+1},\dots,\eps_{j_m},\eps_{j_m+1},\dots,\eps_T\}\!\} \notag \\
\vdash\{\!\{\mu*((m+2)\cdot \lceil\frac{x}{8\elx^5}\rceil),\rho_0,\dots,\rho_m,\eps_{j_m+1},\dots,\eps_T \}\!\} \notag \\
\vdash\{\!\{\mu*((m+2)\cdot \lceil\frac{x}{8\elx^5}\rceil-\sum_{i=1}^m\lceil\frac{\rho_i}{\mu}\rceil), \notag  \\
\rho_0-\mu\cdot\lfl\frac{\rho_0}{\mu}\rfl,\dots,\rho_m-\mu\cdot\lfl\frac{\rho_m}{\mu}\rfl,\eps_{j_m+1},\dots,\eps_T \}\!\} \notag \\
 \vdash\{\!\{ \mu*\lceil\frac{x}{8\elx^5}\rceil,1,2,\dots,2^m,\eps_{j_m +1},\dots\eps_T\}\!\}
\end{gather}
We can see that by $V\deq j_m$, desired properties for $V$ are satisfied.
 \end{proof}
Now, we can prove Theorem \ref{first step} , \ref{primestointerval}
(The author will restate these theorems.)
 Theorem 3.2 can be proved by simply applying Lemma \ref{rtov},\ref{stote}, and \ref{mutomu1}. 
\primeint*
\begin{proof}[proof of Thm\ref{primestointerval}]
Assume $x>L$, $\pset\subset\ppx$ and $|\pset|>(2x/\elx^2)+9$ are satisfied.
We can construct $\pset_1,\pset_2$ such that $\pset = \pset_1\sqcup\pset_2$,$|\pset_1|\ge|\pset_2|\ge(x/\elx^2)+4$ are satisfied by choosing elements of $\pset$.
By Lemma \ref{rtov} , there exist $M$(and $\mathscr{M}\deq\lfl\log_2(M-0.1)\rfl$) such that
\begin{equation}
\{\pset_1\}\vdash\{\!\{M*\lceil\frac{x}{8\elx^4}\rceil\}\!\}.
\end{equation}
By Lemma \ref{stote} , there exists a number sequence  $T_0,\dots,T_\mathscr{M}$,$E_1,\dots,E_S$ such that $T_i\equiv2^i \mod M$ , $T_i<2(M-1)x$ and $M|E_i$,$\sum E_i>x/648\elx^{18}$ ,$E_i<x/8\elx^5$ are setisfied and
\begin{equation}
    \{\pset_2\}\vdash\{\!\{T_0,\dots,T_\mathscr{M},E_1,\dots ,E_S   \}\!\}
\end{equation}
are satisfied.
By considering $\sum E_i>x/648\elx^{18}>4\elx^3x>2(M-1)(\mathscr{M}+1)+4x$, we can see that
$T_0,\dots,T_\mathscr{M},E_1,\dots ,E_S$ satisfy conditions required for $\tau_\bullet,\eps_\bullet$ in Lemma \ref{mutomu1}.
We can conclude
\begin{gather}
    \{\pset\}=\{\pset_1\}+\{\pset_2\} \notag \\
    \vdash\{\!\{M*\lceil\frac{x}{8\elx^4}\rceil,T_0,\dots,T_\mathscr{M},E_1,\dots,E_S\}\!\} \notag \\
    \vdash \{\!\{M*\lceil\frac{x}{8\elx^5}\rceil,1,2,\dots\,2^\mathscr{M},E_V,\dots ,E_S\}\!\}.
\end{gather}
By definition of $\vdash$, there exists $\cale  $ such that
\begin{gather}
    \sumcirc(\{\pset\})\supset\cale+\sumcirc(\{\!\{M*\lceil\frac{x}{8\elx^5}\rceil,1,2,\dots\,2^\mathscr{M},E_V,\dots ,E_S\}\!\})  \notag \\
    =\cale+I(\sum_{i=0}^\mathscr{M}2^i \ +M\cdot \lceil\frac{x}{8\elx^5}\rceil +\sum_{i=V}^S E_i ).
\end{gather}
last equality sign comes from the fact that 
\begin{equation}
1,2,\dots,2^\mathscr{M},M,\dots(\lceil x/8\elx^5 \rceil times)\dots,M,E_V,\dots,E_S
\end{equation}
is SGS and lemma \ref{intervalfromsgs}.

\end{proof}
Theorem\ref{first step} easily follows from Theorem\ref{primestointerval}
\fstep*
\begin{proof}[Proof of Theorem\ref{first step}]
    It follows from Theorem\ref{primestointerval} that there exist $\cale_1,\calf_1,\cale_2,\calf_2$ such that
    $\cale_1,\cale_2<2x^2\ ,\ \calf_1,\calf_2>4x$ and
    \begin{equation}
        \sumcirc(\qset)\supset\cale_1+I(\calf_1)\ ,\  \sumcirc(\qset\setminus\{p\})\supset\cale_2+I(\calf_2) \notag
    \end{equation}
    are satisfied. So
    \begin{gather}
        \sum_{d\in\mathscr{D}}[d]= \sumcirc(\qset) +p\cdot  \sumcirc(\qset\setminus\{p\}) \notag \\
\supset \cale_1+I(\calf_1) +p\cdot(\cale_2+I(\calf_2)) \notag \\
=(\cale_1+p\cale_2)+I(p\cdot\calf_2+\calf_1). \notag
        \end{gather}
        We can see that $\cale_1+p\cale_2<4x^3$ and $p\cdot\calf_2+\calf_1>4x^2$ 
        so these numbers satisfy conditions required for $E,F$.
\end{proof}

We will define the notion of normality of the set, this notion plays an important role in the last part of the 
proof of Theorem \ref{logicallymainthm}.
\begin{defi}
Let $\{x_i\}_{i=1}^n$ be a finite sequence of real numbers.
We say $\{x_i\}_{i=1}^n$ is \textbf{normal} when
\begin{equation}\label{defnormal}
    x_2\le2x_1\ ,\ x_k\le\sum_{i=1}^{k-1}x_i\ (3\le k\le n)\ ,\  
    \sum_{i=1}^n x_n>1.
\end{equation}
Let $C\deq\{c_1,c_2,\dots,c_w\}\subset\nat\ (c_i<c_{i+1}\ 
for\ 1\le i \le w-1)$,
We say that $C$ is \textbf{conormal} when sequence of real numbers
$1/c_w,1/c_{w-1},\dots,1/c_2,1/c_1$ 
satisfies $(\ref{defnormal})$.
\end{defi}
\begin{lma}\label{extnormal1}
(1) Assume that $C$ and $D$ is set of integers
and $C$ is conormal.
\begin{equation}
C\subset D\subset C\cup\{c_{w}+1,c_{w}+2,\dots,2c_{w-1}-1,2c_{w-1}\}, \notag
\end{equation}
then $D$ is also conormal.

(2) If $C$ is cpnormal and $|D\setminus C|\ge 2$ and

\begin{equation}
C\subset D\subset C\cup\{c_{w}+1,c_{w}+2,\dots,2c_{w}-1,2c_{w}\}, \notag
\end{equation}
then $D$ is also conormal.
\end{lma}
\begin{proof}
We prove (2) first.
Assume that 
\begin{equation}
C=\{c_1,c_2,\dots,c_w\}\subset\nat\ (c_i<c_{i+1} 
for\ 1\le i \le w-1) \notag
\end{equation}
is conormal and
\begin{equation}
D=\{c_1,c_2,\dots,c_w,d_1,d_2,\dots d_v\}\ (v\ge 2) \notag
\end{equation}
and 
\begin{equation}
c_1<c_2<\cdots<c_{w-1}<c_w<d_1<\cdots<d_{v-1}<d_v\le 2c_w. \notag
\end{equation}
By assumption we gain
\begin{gather}\notag
1/d_{v-1}<2/2c_w\le2/d_v,\\ \notag
1/d_{v-k}<1/c_w=1/2c_w+1/2c_w<1/d_{v-1}+1/d_v\ (2\le k\le v-1),\\ \notag
1/c_{w-1}\le 2/c_w<1/c_w+1/d_{v-1}+1/d_v,\\ \notag
1/c_{w-k}\le1/c_w+1/c_{w-1}+\cdots+1/c_{w-k+1}\ (2\le k\le w-1),\\ \notag
\sum 1/d_\bullet+\sum1/c_\bullet>\sum1/c_\bullet>1.
\end{gather}
These inequalities mean that $D$ is normal. So Lemma (2) is proved.
When $|D\setminus C|\ge 2$, Lemma (1) follows from
Lemma (2). So it is enough to show that Lemma (1) when $|D\setminus C|=1$.
Assume $C=\{c_1,\dots,c_w\}$ is conormal and let $D=\{c_1,c_2,\dots,c_w,d\}$ with 
\begin{equation}
c_1<c_2<\cdots <c_{w-1}<c_w<d\le 2c_{w-1}. \notag
\end{equation}
By assumption we gain
\begin{gather}\notag
1/c_w<2/2c_{w-1}\le 2/d,\\ \notag
1/c_{w-1}=1/2c_{w-1}+1/2c_{w-1}<1/d+1/c_w,\\ \notag
1/c_{w-k}\le1/c_w+\cdots+1/c_{w-k+1}\ (2\le k\le w-1),\\ \notag
1/d +\sum 1/c_\bullet>\sum 1/c_\bullet>1.
\end{gather}
These inequalities mean that $D$ is conormal. So Lemma (1)
is proved.
\end{proof}

\begin{prop}\label{inverse} 
Let $B = \{b_1,b_2,\dots,b_m\}\subset\nat 
\ (b_1<b_2<\cdots<b_m)$ 
be a finite set that satisfies the conditions below.

1.$L<\min(B) = b_1$

2.\begin{equation}
\sum_{b\in B}\frac{1}{b} >1+\frac{9}{2L}
=1+\frac{1}{2L}+\frac{2}{L}+\frac{1}{L}+\frac{1}{2L}+\cdots \notag
\end{equation}
Then there exist $1\le l\le m$ such that
$\{b_1,b_2,\dots,b_l\}$ is conormal.
\end{prop}
\begin{proof}
Our goal is to show that if ``$B$ satisfies condition 1 and there is $\mathbf{no}$ $1\le l\le m$
such that $\{b_1,b_2,\dots,b_l\}$ is conormal'',
then ``condition 2 fails''.
Assume that $B$ satisfies the assumption above. Let
\begin{equation}
K\deq \min(\{b_k\ |\ \sum_{i=1}^k1/b_i>1\})
\end{equation}

(Condition 2 fails if the set on the right hand side is empty.)
Note that $K>2L$ and $1< \sum_{i=1}^k1/b_i<1+1/2L$.

Let $K_j\deq 2^jK$ and $\cali_j\deq(K_j,K_{j+1}]\ (j\in\nat)$,
and $B_n\deq\cali_n\cap B\ (n\in\nat)$,
$B_{-1}\deq B\cap[0,K]$ , $B_{\le n}\deq\bigcup_{i=-1}^{n}B_i$. Define
\begin{equation}
S_n\deq\sum_{b\in B_n}\frac{1}{b}. \notag
\end{equation}
If $i\ge n(0)$ then $S_i<1/L$ because if not, then $B_{\le i}$ is conormal and it violates the assumption.
Let $n(0)\deq0$ and 
$n(1)\deq \min(\{m\ |\ \sum_{i=n(0)}^{m-1}S_i>1/L\})$. (Condition 2 fails if the set on the right hand side is empty , then $\sum1/b<1+1/2L+1/L$ .)
Note that $\sum_{i=n(0)}^{n(1)-1}S_i<2/L$.

If $i\ge n(1)$ then $S_i < 1/K_{n(0)}$ because if not, then
$B_{\le i}$ is conormal and it violates the assumption.
 Let $n(2)\deq \min(\{m\ |\ \sum_{i=n(1)}^{m-1}S_i>1/K_{n(0)}\})$.
(Condition 2 fails if the set on the right hand side is empty.) Note that
$\sum_{i=n(1)}^{n(2)-1}S_i<2/K_{n(0)}$.

If $i \ge n(2)$ then $S_i<1/K_{n(1)}$
because if not, then $B_{\le i}$ is conormal and it violates the assumption.
Let $n(3)\deq \min(\{m\ |\ \sum_{i=n(2)}^{m-1}S_i>1/K_{n(1)}\})$.
(Condition 2 fails if the set on the right hand side is empty.) Note that
$\sum_{i=n(2)}^{n(3)-1}S_i<2/K_{n(1)}$.

 Because $B$ is finite. We can continue this argument until
 \begin{equation}\label{knz}
 \{m\ \big|\ \sum_{i=n(Z)}^{m-1}S_i>1/K_{n(Z-1)}\}=\emptyset
 \end{equation}
holds.($Z$ is smallest integer such that \ref{knz} is satisfied.)  This means condition 2 fails because
\begin{gather}
\sum_{b\in b}\frac{1}{b}=\sum_{i=-1}^\infty S_i<1+\frac{1}{2L}+\Big(\sum_{i=0}^{Z-2} \frac{2}{K_{n(i)}}\Big)+\frac{1}{K_{n(Z-1)}}  \notag
\\\le1+\frac{1}{2L}+\frac{2}{L}+\frac{1}{L}\cdots+\frac{1}{2^{Z-2}L}<1+\frac{9}{2L}.\notag
\end{gather}

(note that $n(t)\ge t$) 

\end{proof}

\section{Proof of main theorem}

Recall that we fixed $L \deq 8.371\times10^{45}$.
Let $L_k\deq 2^kL$.

For $k\in\nat$, define $J_k\deq[L_k,L_{k+1})$.

\begin{thm}\label{logicallymainthm}
Assume that the set of prime numbers  $\pset$ satisfies conditions 1, 2, and 3 below.

1.$L<\min(\pset)$

2.$|J_k\cap\pset| = 0 $ or $|J_k\cap\pset|>2L_k/(\log L_k)^2 + 10$
for all $k\in \nat$.

3.
\begin{equation}
\sum_{p\in\pset}\frac{1}{p} >1+\frac{9}{2L} \notag
\end{equation}
Then $N\deq \Pi(\pset)$ is semiperfect.

\end{thm}
\begin{proof}
Let $ \pset_m\deq \pset\cap J_m $. By Proposition \ref{inverse}, there exists
\begin{equation}
p_{u}\in \pset_T\deq
\{p_1,p_2,\dots,p_l\}
\ (T\in\nat)\ (p_i<p_{i+1} for\ 1\le i\le l-1) \notag
\end{equation}
 such that 
$\pset\cap[0,p_u]$ is normal.
Define $\pset_{<T}\deq\pset\cap[0,L_k)\deq\{q_1,\dots q_W\}
\ (q_i<q_{i+1}\ for\ 1\le i\le W-1)$ (\textbf{Notice}; We will use the fact that $|\pset_T|=l>2L_T/(\log L_T)^2+10$ follows from condition 3 frequently.)

From Lemma (1), (2).
(If $u=1$, apply Lemma (2), if not, then apply Lemma (1).),
Set gained by adding rest of elements in $\pset_{T}$,
say
\begin{equation}
\pset'\deq\pset_{<T}\cup \pset_T= \{q_1,q_2,\dots,q_W,p_1,p_2,\dots p_l\} \notag
\end{equation}
is also normal. We will prove that $N'=\Pi(\pset')$ is semiperfect.
And $N$ is. 

 (O) Our goal is to extend the interval of length   $\calf>4L_T^2$ ($\calf$ will be defined later) derived from Theorem \ref{first step} by adding $[d] =\{0,d\}$  ($d$: proper divisors of $N'$,$d\not\in\pset_T\cup p_1\cdot(\pset_T\setminus\{p_1\})$) while keeping the shape of interval until its length exceeds $N'$. 

  To achieve this, we must make a slowly growing sequence with
$\calf$ first and followed by proper divisors of $N'$. (Not element of $\pset_T\cup p_1\cdot(\pset_T\setminus\{p_1\})$, used in Theorem \ref{first step}.)
 
(I) By substituting $L_T$ for $x$
and $\pset_T$ for $\qset$,
Conditions of Theorem \ref{first step} is satisfied.
And we obtain
\begin{equation}
    (\sum_{i=1}^l[p_i])+(\sum_{j=2}^l[p_1p_j])
    \supset\mathcal{E}+I(\mathcal{F}) \notag
\end{equation}
with $\cale<4L_T^3<N'$ , $\calf>4L_T^2>p_{l-1}p_l$
, we can take in all elements of $\binom{\pset\setminus\{ p_1\}}{2}_\Pi$ to SGS.
(extending SGS by taking in $p_ip_j$ corresponds to expanding integer interval by adding $[p_ip_j]$.)
And we can take in $p_1p_2p_3$ to SGS because
$\sum\binom{\pset\setminus\{ p_1\}}{2}_\Pi>
\binom{l-1}{2}\cdot p_2p_3>p_1p_2p_3.$

(II) Assume $3\le d\le l-2$. If we can take
$p_1p_2\cdots p_d=\min( \binom{\pset_T}{d}_{\Pi})$ in the slowly growing sequence,
because there is a $\prec$-chain of divisors of $N'$
that connects the number of both ends:
\begin{gather}\notag
p_1p_2\cdots p_d\\ \notag
\prec p_1p_2\cdots p_{d-1}p_l\\ \notag
\prec p_1p_2\cdots p_{l-1}p_l\\ \notag
\vdots\\ \notag
\prec p_1p_{l-d+2}\cdots p_{l-1}p_l\\ \notag
\prec p_{l-d+1}p_{l-d+2}\cdots p_{l-1}p_l.
\end{gather}

We can take $p_{l-d+1}p_{l-d+2}\cdots p_l
=\max(\binom{\pset_T}{d}_{\Pi})$(Lemma\ref{sim} repeatedly.)

and all elements in $\binom{\pset_T}{d}_{\Pi}$ into the 
SGS.
And we can take
$p_1p_2\cdots p_{d+1}=\min( \binom{\pset_T}{d+1}_{\Pi})$
into SGS because
$\sum\binom{\pset_T}{d}_{\Pi}>\binom{l}{d}p_1\cdots p_d>p_1p_2\cdots p_{d+1}$.
We continue to apply this argument until the number of prime divisors
reaches $\lfl l/2\rfl$.

(III) In ``center part'' of divisor set,
Let
\begin{equation}
    \underline{\pset_T}\deq\{p_1,\dots,p_{\lfl l/2\rfl}\},
\overline{\pset_T}\deq\{p_{l-\lfl l/2\rfl+1},\dots,p_l\}. \notag
\end{equation}
 We can take in $\Pi(\undPT) $ by the argument in (II), we can gain
\begin{equation}
\frac{\Pi(\overPT)}{\Pi(\undPT)}
>(1+\frac{\lfl l/2\rfl}{2L_T})^{\lfl l/2 \rfl}
>2^{\lfl l/2\rfl\lfl l/2 \rfl/2L_T}
>2^{L_T/2(\log L_T)^4}
>  L_T \notag
\end{equation}
(The last $>$ follows from (\ref{powofx/l4}))and gain
\begin{equation}
\Pi(\undPT)\prec\dots\prec\Pi(\overPT)>q_W \Pi(\undPT)\ . \notag   
\end{equation}
So we can take in $q_W \Pi(\undPT)$.
Repeating similar argument again and again, 
we can take $\Pi(\pset_{<T})\Pi(\undPT)$ in a slowly growing sequence.

(IV) By the same argument as used in (II). We can take $N'/p_l$ in SGS.

(V) Because $\pset'$ is normal,
we can extend the SGS by adding  

$N'/p_{l-1},\dots,N'/p_1,N'/q_W,\dots N'/q_1$.
Now, the sum of SGS ($\deq\calf'$) is larger than $N'$
and $\cale<N'$, which means that 
\begin{equation}
N'\in\cale +I(\calf')=\cale+I(\calf)+\sum_{d ;appeaed\ in\ SGS}[d]   \subset\sum_{d|N',d\ne N'}[d]
\ (\calf'>N'). \notag
\end{equation}
and $N'$ is semiperfect.
And $N$ is.
\emptyline

To summarize, our slowly growing sequence made by $\calf$ and proper divisors of
$N'$ 
is as follows;

\begin{gather}\notag
\calf,
p_{l-1}p_l,\dots \binom{\pset_T\setminus\{p_1\}}{2}_\Pi\dots\\ \notag
\dots,p_1p_2p_3,\dots,p_{l-2}p_{l-1}p_l,\dots
\binom{\pset_T}{3}_\Pi\dots\\ \notag
\dots\binom{\pset_T}{4}_\Pi
\dots\dots\binom{\pset_T}{\lfl l/2\rfl-1}_\Pi\dots\\ \notag
\dots,\Pi(\undPT),\dots,\Pi(\overPT),q_1 \Pi(\undPT),\dots\\ \notag
\dots,\Pi(\pset_{<T})\Pi(\overPT),\dots\dots\Pi(\pset_{<T})\binom{\pset_T}{l-2}_\Pi\dots\\ \notag
\dots,N'/p_l,\dots, N'/p_1, N'/q_W,\dots ,N'/q_1.
\end{gather}

($\cdots X\cdots$ means lining up of elements of $X$ did not appear before 
by ascending order.)

\end{proof}

For a finite set $S\in \nat $.
Let $R(S)\deq \sum_{p\in S}1/p$.

\begin{proof}[Proof of Theorem $\ref{mainthm633}$]
By Theorem 6.10 from \cite{Dus2010},
\begin{equation}\label{invofsmallp}
\sum_{p:prime,p<L}\frac{1}{p}
<\log\log L\ + 0.26150 +\frac{1}{10(\log L)^2}+\frac{1}{15(\log L)^3}
<4.92252. \notag
\end{equation}

Assume $\sigma(n)/n>633$.
Then
\begin{equation}
    \frac{\sigma(n)}{n}\cdot
    \Pi_{p\in\pset(n)}(1+\frac{1}{p})^{-1}
    <\Pi_{p\in\pset(n)}(1+\frac{1}{p^2}+\frac{1}{p^4}+\cdots)
    <\frac{\pi^2}{6}. \notag
\end{equation}
So
\begin{equation}
\Pi_{p\in\pset(n)}(1+\frac{1}{p})
>\frac{\sigma(n)}{n}\cdot\frac{6}{\pi^2}
>633\cdot\frac{6}{\pi^2}>384.8178. \notag
\end{equation}
Let $\pset(n)$ be the set of the prime factors of $n$.
By omitting primes smaller than $L$,
``thin'' part, ``too large'' part, we can obtain a set of primes $\pset_{\mathrm{good}}$, which satisfies the conditions of Theorem \ref{logicallymainthm}.
We see that
\begin{equation}
\sum_{p\in\pset(n)}\frac{1}{p}
>\log(\Pi_{p\in\pset(n)}(1+\frac{1}{p}))
>\log(384.8178)
>5.95276. \notag
\end{equation}
Let $\pset(n)_{>L}\deq\pset(n)\cap(L,\infty)$,
\begin{equation}
\sum_{p\in \pset(n)_{>L}}\frac{1}{p}
\ge(\sum_{p\in\pset(n)}\frac{1}{p})\ - \sum_{p:prime,p<L}\frac{1}{p}
>5.95276-4.92252=1.03024, \notag
\end{equation}

Define $N_{\mathrm{omit}}\deq\{i\in\nat\ \big|\ |\pset(n)_{>L}\cap J_i|<(2L_i/(\log L_i)^2)+10\}$.
And define $\pset_{\mathrm{omit}}\deq\bigcup_{i\in N_{\mathrm{omit}}}( \pset(n)_{>L}\cap J_i )$.
The sum of the reciprocals of the elements of $\pset_{\mathrm{omit}}$ is bounded by
\begin{gather}\notag
\sum_{p\in\pset_{\mathrm{omit}}}\frac{1}{p}
<\sum_{k\in N_{\mathrm{omit}}}\frac{2L_k/(\log L_k)^2 + 10}{L_k}\\ \notag
<\sum_{k = 152}^\infty \frac{2}{(k\cdot \log 2 )^2}
+\sum_{k=1}^{\infty}\frac{10}{L_k}
<\frac{2}{(\log 2)^2}\cdot\frac{1}{151}+10^{-40}
<0.02757.
\end{gather}
(Second $<$ follows from $L>2^{152}$.)
Let $\pset_{\mathrm{good}}\deq \pset(n)_{>L}\setminus\pset_{\mathrm{omit}}$.
Its sum of reciprocals of its elements is bounded by
\begin{equation}
\sum_{p\in\pset_{\mathrm{good}}}\frac{1}{p}>1.03024-0.02757=1.00267>1+\frac{9}{2L}. \notag
\end{equation}
By definition , $\pset_{\mathrm{good}}$ satisfies the conditions required in 
Theorem \ref{logicallymainthm}. 
It means $\Pi(\pset_{\mathrm{good}})$ is semiperfect and $n$ is.
\end{proof}
We will estimate how big $n$ will be if condition
of \ref{mainthm633} holds.
Under the Riemann Hypothesis,
if $l(n)>633$ , 
Then by  Robin's theorem,
\begin{equation}
    633<l(n)<e^\gamma \log\log n \notag
\end{equation}
holds.
The smallest integer satisfying $\sigma(n)/n>633$
is larger than
\begin{equation}
   \exp(\exp(\frac{633}{e^\gamma})) \approx 10^{9.72\times10^{153}}. \notag
\end{equation}

%comlile overleaf
\bibliographystyle{amsplain}
\bibliography{weird}

%arXiv  post
%\input{output.bbl}

\end{document}